\documentclass{amsart}

\usepackage{amsmath,amssymb,amsthm}
\usepackage{url}
\usepackage{bbm}
\usepackage{graphicx,tikz}
\usetikzlibrary{shapes,arrows}
\usetikzlibrary{calc, positioning, shapes, through, intersections}
\newtheorem{theorem}{Theorem}
\newtheorem*{theorem*}{Theorem}
\newtheorem*{lemma*}{Lemma}
\newtheorem*{proposition}{Proposition}

\theoremstyle{definition}

\theoremstyle{remark}

\makeatletter
\def\Let@{\def\\{\notag\math@cr}}
\makeatother

\begin{document}
\title[]{Transport and Interface: an uncertainty\\ principle for the Wasserstein distance}
\keywords{Wasserstein, Uncertainty principle, Nodal set, Sturm-Liouville theory.}
\subjclass[2010]{28A75, 58C40, 90B06} 

\author[]{Amir Sagiv}
\address{Department of Applied Mathematics, Tel Aviv University, Tel Aviv 6997801, Israel}
\email{amirsagiv@mail.tau.ac.il}

\author[]{Stefan Steinerberger}
\address{Department of Mathematics, Yale University, New Haven, USA}
\email{stefan.steinerberger@yale.edu}
\thanks{S.S. is supported by the NSF (DMS-1763179) and the Alfred P. Sloan Foundation. This research was carried out during a long-term stay of A.S.  at Yale University whose hospitality is gratefully acknowledged.}

\begin{abstract} Let $f: [0,1]^d \rightarrow \mathbb{R}$ be a continuous function with zero mean and interpret $f_{+} = \max(f, 0)$ and $f_{-} = -\min(f, 0)$ as the densities of two measures. We prove that if the cost of transport from $f_{+}$ to $f_{-}$ is small (in terms of the Wasserstein distance $W^1$), then the  nodal set $\left\{x \in (0,1)^d: f(x) = 0 \right\}$ has to be large (`if it is always easy to buy milk, there must be many supermarkets'). More precisely, we show that
$$ W_1(f_+, f_-) \cdot \mathcal{H}^{d-1}\left\{x \in (0,1)^d: f(x) = 0 \right\} \gtrsim_{d}  \left( \frac{\|f\|_{L^1}}{\|f\|_{L^{\infty}}} \right)^{4 - \frac1d} \|f\|_{L^1} \, .$$
We apply this ``uncertainty principle" to the metric Sturm-Liouville theory in higher dimensions to show that a linear combination of eigenfunctions of an elliptic operator cannot have an arbitrarily small zero set.
\end{abstract}

\maketitle

\section{Introduction and Results}
\subsection{Introduction.}  This paper is concerned with a basic question in measure theory: let $\Omega = [0,1]^d$ and let $f:\Omega \rightarrow \mathbb{R}$ be a continuous function with zero mean. This induces two absolutely continuous measures
$$ \mu =  \max(f, 0)\, dx \qquad \mbox{and} \qquad \nu = -\min(f, 0)\, dx \, .$$
We may think of $f_+$ (or $\mu$) as the `surplus' and of $f_{-}$ (or $\nu$) as the `deficit'. How expensive is it to transport the surplus to the deficit, and how is this transport related to the geometry of the function $f$? 

\begin{center}
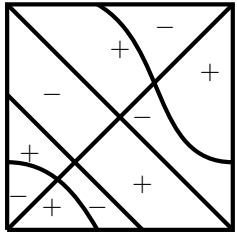
\begin{figure}[h!]
\begin{tikzpicture}[scale=3]
\draw [ultra thick] (0,0) -- (1,0) -- (1,1) -- (0,1) -- (0,0);
\draw [ultra thick] (0, 0.3) to[out=0, in=120] (0.4, 0);
\draw [ultra thick] (0.4, 1) to[out=330, in=180] (1, 0.3);
\draw [ultra thick] (0,1) -- (1,0);
\draw [ultra thick] (0,0.6) -- (0.6,0);
\draw [ultra thick] (0,0) -- (1,1);
\node at (0.2, 0.1) {$+$};
\node at (0.1, 0.34) {$+$};
\node at (0.05, 0.15) {$-$};
\node at (0.6, 0.2) {$+$};
\node at (0.9, 0.7) {$+$};
\node at (0.7, 0.9) {$-$};
\node at (0.2, 0.6) {$-$};
\node at (0.5, 0.8) {$+$};
\node at (0.6, 0.5) {$-$};
\node at (0.4, 0.1) {$-$};
\end{tikzpicture}
\caption{The sign pattern of a function $f:[0,1]^2 \rightarrow \mathbb{R}$, defining the measures $\mu$ and $\nu$ on the supports of $f_{\pm}$, respectively.  The boundary lines are the nodal sets $\left\{x: f(x) = 0\right\}$ .}
\end{figure}
\end{center}

For this question to be meaningful, we use the Wasserstein distance as a notion of 'cost of transport' between measures.
The Wasserstein distance was introduced in the late 1960's \cite{dob, wasser} and is now a fundamental notion in the theory of Optimal Transport \cite{villani}. The $p-$Wasserstein distance between two measures $\mu$ and $\nu$ on a domain $\Omega$ is defined via
$$ W_p(\mu, \nu) = \left( \inf_{\gamma \in \Gamma(\mu, \nu)} \int_{\Omega \times \Omega}{ |x-y|^p \, d \gamma(x,y)}\right)^{1/p},$$
where $| \cdot |$ is the distance and $\Gamma(\mu, \nu)$ denotes all measures on $\Omega \times \Omega$
with marginals $\mu$ and $\nu$ (couplings of $\mu$ and $\nu$). We will mainly work with $p=1$, also known as Earth Mover's Distance, which has the most immediate physical interpretation as $\mbox{cost} = \mbox{mass} \times \mbox{distance}$.
Our paper is motivated by a simple intuition.
\begin{quote}
If it is cheap to transport $\mu$ to $\nu$, then most of the positive mass must be somewhat close to most of the negative mass. Hence, the supports of $\mu$ and $\nu$ alternate frequently and so the nodal set $\left\{x \in \Omega: f(x) = 0 \right\}$ should be large.
\end{quote}
This notion was first formulated by the second author in \cite{stein2}, where a two-dimensional version of such a statement was proven and applied for a problem in Sturm-Liouville theory (this is outlined in \S 1.3). The purpose of this paper is to show that such results also exist in higher dimensions.

\subsection{Main Results.} We now state our main result.

\begin{theorem}  Let $f:[0,1]^d \rightarrow \mathbb{R}$ be a continuous function with mean value 0. Then, with an implicit constant depending only on $d$,
$$ W_1(f_+, f_-) \cdot \mathcal{H}^{d-1}\left\{x \in (0,1)^d: f(x) = 0 \right\} \gtrsim_{d}  \left( \frac{\|f\|_{L^1}}{\|f\|_{L^{\infty}}} \right)^{4 - \frac1d} \|f\|_{L^1} \, ,$$
where $\mathcal{H}^{d-1}$ is the $(d-1)$-dimensional Hausdorff measure.
\end{theorem}

This is the first result of its type in higher dimensions. It is known~\cite{stein2} that for $d=2$ the exponent $4- 1/d = 3.5$ can be replaced by $1$, but that argument is very different since it exploits a property of two dimensions that cannot be generalized to higher dimensions. We also note that Theorem 1 can easily be extended to the $W_p$ distance for $p>1$, see~\S \ref{sec:wp}. A version of Theorem 1 should hold on more general bounded domains. However, since $W_1(\mu ,\nu)\leq {\rm diam} (\Omega) \cdot \|f\|f_1$, it is necessary that the domain is one in which a relative isoperimetric inequality is valid, see discussion in \S 5.4. We expect a form of the statement to hold on convex domains.
While we do not know what the sharp form of the Theorem in $(0,1)^d$ could be, we show that the exponent $4-1/d$ cannot be replaced by anything smaller than $1-1/d$.
\begin{proposition} An estimate of the form
 $$ W_1(f_+, f_-) \cdot \mathcal{H}^{d-1}\left\{x \in (0,1)^d: f(x) = 0 \right\} \gtrsim_{\Omega}  \left( \frac{\|f\|_{L^1}}{\|f\|_{L^{\infty}}} \right)^{\alpha} \|f\|_{L^1}$$
 can only hold for $\alpha \geq (d-1)/d$.
 \end{proposition}

The construction in the Proposition easily generalizes to other domains. Besides the general elementary proof, we also give a spectral proof for the sphere $\Omega = \mathbb{S}^{d}$, adapted from \cite{stein2}, that makes use of the theory of spherical $t-$designs \cite{bond}. The advantage of such a purely spectral argument is that it might be used to prove possible discrete versions of the problem, which we discuss in \S 5.

 \subsection{Applications to Sturm-Liouville Theory.} 
Sturm-Liouville theory \cite{liouville, sturm, sturm2} shows that the eigenfunctions of an elliptic operator on the real line are rather rigid. More precisely, consider an operator of the type
$$ H = -\frac{d}{dx}\left(a(x) \frac{d}{dx} \right) + b(x) \qquad \mbox{on an interval}~(a,b)$$
with Neumann or Dirichlet boundary conditions, where $a(x), b(x) > 0$ are bounded away from 0, then there is a discrete sequence of eigenfunctions $\phi_k$ that form a basis of $L^2(a,b)$ where the eigenfunction $\phi_k$ has $k-1$ roots. An even stronger statement is that
$$ \sum_{k \geq n}{a_k \phi_k} \qquad \mbox{has at least}~n-1~\mbox{roots}.$$
Phrased differently, adding functions that oscillate at higher frequencies cannot decrease the number of roots. This result is often attributed to Hurwitz \cite{hurwitz} but is already contained in Sturm's original work \cite{sturm, sturm2}, see B\'{e}rard \& Helffer \cite{berard, berard2} for a fascinating historical overview and a modern proof, and see \cite{stein3} for a recent quantitative result. One could now wonder whether such results are also possible in higher dimensions: given a compact domain, we can consider eigenfunctions of the Laplacian~$\phi_k$ (or more general elliptic operator with Neumann boundary condition) and ask whether the linear combination of the eigenfunctions $\phi _k$ can vanish only on a small set.

\begin{theorem}[\cite{stein1}]
Let $\Omega$ be a compact domain in $\mathbb{R}^d$ or a compact $d-$dimensional manifold, let $(\phi_k)$ denote the eigenvalues of the Laplacian, $-\Delta \phi_k = \lambda_k \phi_k$, with Neumann boundary conditions and let $f:M \rightarrow \mathbb{R}$ denote a continuous function whose spectral decomposition is given by
 $$ f = \sum_{\lambda_k \geq \lambda}{ \left\langle f, \phi_k \right\rangle \phi_k}\, ,$$
then
$$ \mathcal{H}^{d-1} \left\{x:f(x) =0\right\} \gtrsim_{\Omega} \frac{  \sqrt{\lambda}}{(\log{\lambda})^{d/2}} \left( \frac{ \|f\|_{L^1}}{\|f\|_{L^{\infty}}} \right)^{2 - \frac{1}{d}}\, .$$
\end{theorem}
The scaling has the optimal dependance, up to logarithms, on $\lambda$ (see \cite{stein2}).
The proof in \cite{stein1} makes explicit use of the heat flow and combinatorial arguments. If we assume some regularity on the nodal set, then stronger results are possible. Theorem 1  (together with another observation, see \cite[\S 9]{steiner2018wass} and \S \ref{sec:upbd_pf}) implies a  Sturm-Hurwitz type result in the case of the unit cube. 

\begin{theorem}
Let $\Omega=(0,1)^d$ and assume the setting of Theorem 2. Then 
$$ \mathcal{H}^{d-1} \left\{x:f(x) =0\right\} \gtrsim_{\Omega} \frac{  \sqrt{\lambda}}{\log{\lambda}} \left( \frac{ \|f\|_{L^1}}{\|f\|_{L^{\infty}}} \right)^{4 - \frac{1}{d}}.$$
\end{theorem}
This establishes a strong connection between our measure-theoretic question (Theorem 1) and the theory of elliptic partial differential equations. 
While this result is weaker than Theorem 2, it illustrates one striking way in which Theorem~1 can be applied. Moreover, any improvement of the bound in Theorem~1 immediately implies a corresponding improvement of Theorem 3.

\section{Proof of Theorem 1}

\subsection{Outline.}
We assume that $f:[0,1]^d \rightarrow \mathbb{R}$ is continuous and has mean value
$$ \int_{\Omega}{f(x)dx} = 0 \, .$$
We introduce the two functions $f_{+}(x) = \max(f(x), 0)$ and $f_{-}(x) = -\min(f(x), 0)$ and a partition of $\Omega$ into cubes at scale $\varepsilon$, denoted by $\left\{Q_j\right\}_{j=1}^{\varepsilon ^{-d}}$.
Our main idea is to decouple the problem into many small cubes: we show that many cubes contain a non-negligible portion of $\|f_+\|_{L^1}$, the $\mu$-mass. In some of these cubes, $f$ is mostly positive, and so transporting $\mu$ from these cubes is expensive. The other cubes support both $\mu$ and $\nu$ in a non-negligible way, and so contain a large interface, i.e., a large nodal set. We will, throughout the proof, write $\|\cdot\|_{L^1(Q_j)}$ and $\|\cdot\|_{L^{\infty}(Q_j)}$ for the norms on a cube $Q_j$ and $\| \cdot\|_{L^1}$ as well as $\|\cdot \|_{L^{\infty}}$ for the global norms on~$[0,1]^d$. Since we do not expect the overall scaling to be optimal, our arguments will all be up to universal constants.
  
\subsection{Dismissing $L^1$ negligible cubes.} The first step consists of dismissing cubes on which $\|f_{+}\|_{L^1(Q_j)}$ is negligible. More precisely, we define
$$\mathcal{A} :\,= \left\{Q_j ~~{\rm s.t.}~~ \varepsilon ^{-d}\|f_+\|_{L^1(Q_j)} \leq  \frac{1}{100}\|f_+\|_{L^1(\Omega)} \right\} \, ,
$$
and $\mathcal{B} = \left\{Q_j\right\} \setminus A$ as the complement.
By definition, since the cubes are disjoint, 
  $$\|f_+\|_{L^1(\mathcal{A})} = \sum\limits_{Q\in\mathcal{A}} \|f\|_{L^1(Q)} \leq  \frac{\varepsilon ^d}{100}\|f_+\|_{L^1(\Omega)}  \cdot |\mathcal{A}|$$
and thus
\begin{align*}
\|f\|_{L^1(\Omega)}  = 2\|f_+\|_{L^1(\Omega)} & \leq\frac{\varepsilon ^d}{50}\|f_+\|_{L^1(\Omega)}  |\mathcal{A}|  +   \varepsilon ^d \|f\|_{L^{\infty}(\mathcal{B})}  |\mathcal{B}|  \\
& \leq \varepsilon ^{-d}\frac{\varepsilon ^d}{50}\|f\|_{L^1(\Omega)} +   \varepsilon ^d \|f\|_{\infty}  |\mathcal{B}| \, ,
\end{align*}
where the first equality follows from $\int f = 0$, and the last inequality follows from the trivial estimate $|\mathcal{A}|  \leq \varepsilon ^{-d}$. Hence $|\mathcal{B}|$, the number of cubes where $f_+$ is not negligible, is bounded from below
$$|\mathcal{B}|>\frac{49}{50}\varepsilon^{-d} \frac{\|f\|_{L^1}}{\|f\|_{L^{\infty}}} \, . 
$$
We will henceforth only work with cubes in $\mathcal{B}$.
\subsection{Balanced and unbalanced cubes} Let us now consider a cube $Q_j \in \mathcal{B}$, i.e. a cube for which
$$ \int_{Q_j}{f_+ dx} \geq \frac{ \varepsilon ^d}{100} \|f_+\|_{L^1(\Omega)}.$$
The previous step showed that there are many such cubes. We will divide them into two separate categories
\begin{enumerate}
\item \textit{unbalanced} cubes where the positive mass outweighs the negative mass by a large factor
$$\|f_{-}\|_{L^1(Q_j)} \leq \|f_{+}\|_{L^1(Q_j)}/100$$
\item and if a cube is not unbalanced, we say it is \textit{balanced}.
\end{enumerate}
Our strategy will be to show that {\em unbalanced} cubes induce a nontrivial amount of transport while {\em balanced} cubes, which have both positive and negative mass in nontrivial amounts, have to contain at least some nontrivial portion of the zero set $\left\{x \in \Omega: f(x) = 0\right\}$. 
\subsection{Transporting mass out unbalanced cubes}\label{sec:out_transport}
How much does it cost to transport the surplus out of an unbalanced cube? The worst case, in terms of a lower bound on the transport cost, is when the the support of $f_{-}$ (or $\nu$) in $Q_j$ is well-mixed with that of $f_+$ (or $\mu$), and can therefore absorb a lot of positive mass. In such a case, the transport distances might be arbitrarily small, and so the transport cost might be  arbitrarily small too. However, since $\|f_{-}\|_{L^1(Q_j)} \leq \|f_{+}\|_{L^1(Q_j)}/100$, the deficit $\nu$ in $Q_j$ can absorb
at most a small portion of the positive $L^1-$mass. The rest has to be transported out of the cube.

\begin{center}
\begin{figure}[h!]
\begin{tikzpicture}[scale=5]
\draw [ultra thick] (0,0) -- (1,0) -- (1,1) -- (0,1) -- (0,0);
\draw [thick, dashed] (0.2, 0.2) -- (0.8, 0.2) -- (0.8,0.8) -- (0.2, 0.8) -- (0.2, 0.2);
\node at (0.5, 0.1) {density is maximal: $\|f\|_{L^{\infty}}$};
\node at (0.5, 0.5) {density is 0};
\node at (1.1, 0) {$Q_j$};
\end{tikzpicture}
\caption{The cheapest transport occurs when the mass has to travel the smallest possible amount: this requires the densest possible configuration as close to the boundary as possible.}
\end{figure}
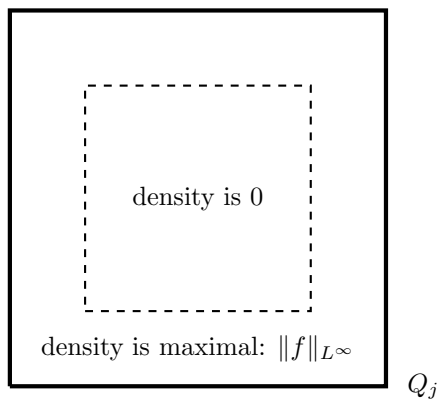
\end{center}

We now try to understand how much it costs to transport mass {\em out} of an unbalanced cube. The lower bound, i.e., the {\em cheapest} transport cost, occurs at the extreme case where the $L^1-$mass is as close as possible to the boundary of $Q_j$. Since the highest possible density of $f_+$ is $\|f\|_{L^{\infty}}$,
the amount of volume that the $L^1-$mass at maximum density occupies is at least $\|f_+\|_{L^1(Q_j)}/\|f\|_{L^{\infty}}$. The thickness $r$ of the '$\ell^{\infty}$-annulus' on the cube's surface is thus given by the solution of the equation
$$ \varepsilon^{d} - (\varepsilon -  2r)^{d} = \frac{ \|f_+\|_{L^1(Q_j)}}{\|f\|_{L^{\infty}}}.$$
The mean-value theorem implies, for some $\varepsilon - r < \xi < \varepsilon$,
$$ \varepsilon^{d} - (\varepsilon -  2 r)^{d} =  2d r \xi^{d-1} \leq d r \varepsilon^{d-1}$$
and thus, for a single unbalanced cube,
$$ r \geq  \frac{1}{2d}\frac{1}{\varepsilon^{d-1}} \frac{ \|f_+\|_{L^1(Q_j)}}{\|f\|_{L^{\infty}}} \, .$$
The total mass to be transported is $\gtrsim \|f_+\|_{L^1(Q_j)}$ and therefore the transport cost out of a single unbalanced cube is
$$ W_1-\mbox{cost} \gtrsim \frac{1}{\varepsilon^{d-1}}\frac{\|f_+\|_{L^1(Q_j)}}{\|f\|_{L^{\infty}}} \|f_+\|_{L^1(Q_j)}  \gtrsim \varepsilon^{d+1} \frac{\|f\|^2_{L^1}}{\|f\|_{L^{\infty}}}\, .$$
where the last inequality is due to $Q_j \in \mathcal{B}$ implying $\|f_+\|_{L^1(Q_j)} \gtrsim \varepsilon^d \|f\|_{L^1(\Omega)}$.

\subsection{Isoperimetry in balanced cubes.} Suppose now that $Q_j \in \mathcal{B}$ is balanced. That means that there is nontrivial amount of $\mu$ mass (at least $\sim \varepsilon^d \|f\|_{L^1(\Omega)}$, since~$Q_j \in \mathcal{B}$). However, there might be comparable amounts of $\nu$ mass (i.e., $\|f_{-}\|_{L^1(Q_j)}$) and it is not obvious how to obtain a nontrivial bound on transportation from $\mu$ in $Q_j$. Instead, in balanced sets, we will estimate the size of the nodal set directly. Indeed, since $\|f_{\pm}\|_{L^1(Q_j)}$ are not too small, their interface, the boundary of the set $\mbox{supp}(f_+)\cap Q_j$ cannot be too small either.  Note that some of this boundary might be in $\partial Q_j$ and is not part of the nodal set. To estimate the size of the nodal set alone, we will use the Relative Isoperimetric Inequality, see Lions \& Pacella \cite{lions}, Morgan \cite{morgan}, Ritore \& Vernadakis \cite{rit} and references therein. 
\begin{theorem*}[Relative Isoperimetric Inequality]
Let $K\subseteq [0,1]^d$. Then $$\mathcal{H}^{d-1}(\partial K \cap (0,1)^d) \gtrsim _d \min \left({\rm vol}(K), {\rm vol}([0,1]^d\setminus K) \right)^{\frac{d-1}{d}} \, .$$
\end{theorem*}
We now apply the relative isoperimetric inequality to the set
$$ K = \left\{x \in Q_j: f(x) > 0\right\}.$$
We observe that, since $Q_j$ is balanced and $Q_j \in \mathcal{B}$,
$$   \mbox{vol}(K) \gtrsim \frac{\|f_+\|_{L^1(Q_j)}}{\|f\|_{L^{\infty}}} \quad \mbox{as well as} \quad \mbox{vol}(Q_j \setminus K) \gtrsim \frac{\|f_-\|_{L^1(Q_j)}}{\|f\|_{L^{\infty}}} \gtrsim \frac{\|f_+\|_{L^1(Q_j)}}{\|f\|_{L^{\infty}}}\, .$$
Thus, using again $\|f_+\|_{L^1(Q_j)} \gtrsim \varepsilon^d \|f\|_{L^1(\Omega)}$,
$$ \mathcal{H}^{d-1}\left( \left\{x \in Q_j: f(x) = 0\right\}\right) \gtrsim \left(  \frac{\|f_+\|_{L^1(Q_j)}}{\|f\|_{L^{\infty}}} \right)^{\frac{d-1}{d}} \gtrsim \varepsilon^{d-1} \left(  \frac{\|f\|_{L^1(\Omega)}}{\|f\|_{L^{\infty}(\Omega)}} \right)^{\frac{d-1}{d}}$$

\subsection{Balancing the scales.} Let us denote the number of balanced cubes by $E$ and the number of unbalanced cubes by $F$. We know that 
$$E+F = |\mathcal{B}|\, .$$
The final ingredient is to show that there is a scale $\varepsilon$ at which most cubes are unbalanced. By adding our estimate over all balanced cubes, we obtain
$$ \mathcal{H}^{d-1}\left( \left\{x \in (0,1)^d: f(x) = 0\right\}\right) \gtrsim \varepsilon^{d-1} \left(  \frac{\|f\|_{L^1(\Omega)}}{\|f\|_{L^{\infty}(\Omega)}}\right)^{\frac{d-1}{d}} E \, .$$
This can be rewritten as
$$ E \lesssim \frac{\mathcal{H}^{d-1}\left( \left\{x \in (0,1)^d: f(x) = 0\right\}\right)}{ \varepsilon^{d-1} \left(  \frac{\|f\|_{L^1(\Omega)}}{\|f\|_{L^{\infty}(\Omega)}}\right)^{\frac{d-1}{d}}}\,.$$
We want to pick $\varepsilon$ such that $E$ is smaller than $|\mathcal{B}|/2$, which then yields $F \sim |\mathcal{B}|$. We recall our lower bound on $|\mathcal{B}|$
$$ |\mathcal{B}| \gtrsim \varepsilon^{-d} \frac{\|f\|_{L^1}}{\|f\|_{L^{\infty}}}  = G \, ,$$
and will therefore pick $\varepsilon$ so as to ensure $E \leq G/2$. This motivates the scale 
$$ \varepsilon \sim
 \left(  \frac{\|f\|_{L^1(\Omega)}}{\|f\|_{L^{\infty}(\Omega)}}\right)^{2-\frac{1}{d}} \frac{1}{ \mathcal{H}^{d-1}\left( \left\{x: f(x) = 0\right\}\right)}. \qquad (\diamond)$$
We show that this value is admissible, i.e., $0 \leq \varepsilon \leq 1$: the relative isoperimetric inequality implies that
$$  \mathcal{H}^{d-1}\left\{x \in (0,1)^d: f(x) = 0 \right\} \gtrsim  \left( \frac{\|f\|_{L^1}}{\|f\|_{L^{\infty}}} \right)^{\frac{d-1}{d}} \, ,$$
and therefore
$$ \varepsilon \sim  \left(  \frac{\|f\|_{L^1(\Omega)}}{\|f\|_{L^{\infty}(\Omega)}}\right)^{2-\frac{1}{d}} \frac{1}{\mathcal{H}^{d-1}\left( \left\{x: f(x) = 0\right\}\right) }\lesssim   \frac{\|f\|_{L^1(\Omega)}}{\|f\|_{L^{\infty}(\Omega)}} \lesssim 1 \, .$$
For $\varepsilon$ sufficiently small (in particular, smaller than $(\diamond)$ suffices), we have $F \sim |\mathcal{B}|$. In \S \ref{sec:out_transport} we obtained a lower bound on the $W_1$ cost of transporting $\mu$ out of a single unbalanced cube. Since this is a lower bound, we can in this case sum it over {\em all $F$ unbalanced cubes} and estimate 
\begin{align*}
 W_1(f_+, f_-) &\gtrsim  \varepsilon^{d+1} \frac{\|f\|^2_{L^1}}{\|f\|_{L^{\infty}}} F \gtrsim  \varepsilon^{d+1} \frac{\|f\|^2_{L^1}}{\|f\|_{L^{\infty}}} |\mathcal{B}| \\
&\gtrsim \varepsilon^{d+1} \frac{\|f\|^2_{L^1}}{\|f\|_{L^{\infty}}} \varepsilon^{-d} \frac{\|f\|_{L^1}}{\|f\|_{L^{\infty}}} \gtrsim \varepsilon \frac{\|f\|^3_{L^1}}{\|f\|^2_{L^{\infty}}} \, .
\end{align*}
Having chosen $\varepsilon$ as in $(\diamond)$, this results in
$$ W_1(f_+, f_-) \cdot \mathcal{H}^{d-1}\left\{x \in (0,1)^d: f(x) = 0 \right\} \gtrsim_{d}  \left( \frac{\|f\|_{L^1}}{\|f\|_{L^{\infty}}} \right)^{4 - \frac1d} \|f\|_{L^1}  \, ,$$
which concludes the argument.

\subsection{The case of $W_p-$distance.}\label{sec:wp} Our argument can be easily adapted to deal with the more general case of $W_p$ distances. 

\begin{theorem} Under the same assumptions as in Theorem 1, we have, for all $1~\leq~p < \infty$,
$$ W_p(f_+, f_-) \cdot \mathcal{H}^{d-1}\left\{x \in (0,1)^d: f(x) = 0 \right\} \gtrsim_{d, p}  \left( \frac{\|f\|_{L^1}}{\|f\|_{L^{\infty}}} \right)^{3 - \frac1d + \frac1p} \|f\|_{L^1} \, ,$$
\end{theorem}
\begin{proof}
The proof is more or less identical, we note the arising changes here. The first change is in the cost of the transporting superfluous mass out of an unbalanced cube. The same geometric argument now implies
$$ \left( W^p-\mbox{cost} \right)^p \gtrsim_p \left(\frac{1}{\varepsilon^{d-1}}\frac{\|f_+\|_{L^1(Q_j)}}{\|f\|_{L^{\infty}}}\right)^p \|f_+\|_{L^1(Q_j)}\, .$$
The assumption $Q_j \in \mathcal{B}$, implying $\|f_+\|_{L^1(Q_j)} \gtrsim \varepsilon^d \|f\|_{L^1(\Omega)}$ then results in
$$ \left( W^p-\mbox{cost} \right)^p \gtrsim_p \varepsilon^{p+d} \left( \frac{\|f\|_{L^1}}{\|f\|_{L^{\infty}}}\right)^p \|f\|_{L^1} \, .$$
The geometric argument regarding the size of the nodal set is unchanged. We thus obtain
\begin{align*}
 W_p(f_+, f_-)^p \gtrsim \varepsilon^{p} \left( \frac{\|f\|_{L^1}}{\|f\|_{L^{\infty}}}\right)^{p+1} \|f\|_{L^1} \, . 
 \end{align*}
 Plugging in the same value of $\varepsilon$ and taking a $p-$th root results in
 $$  W_p(f_+, f_-) \mathcal{H}^{d-1}\left\{x \in (0,1)^d: f(x) = 0 \right\}  \gtrsim   \left( \frac{\|f\|_{L^1}}{\|f\|_{L^{\infty}}} \right)^{3 - \frac1d + \frac1p} \|f\|_{L^1}.$$
\end{proof}

\section{Proof of the Proposition}
 We wish to show that the $4-1/d$ exponent in Theorem 1 cannot be replaced by an exponent smaller than $1-1/d$. We do so by an explicit counterexample: loosely speaking, we take $n$ points that are as far apart from one another as $n$ points can possibly be, put small bump functions around these points and then compute all relevant properties.

\begin{proof}[Proof.] We can assume w.l.o.g.\ that $|\Omega|=1$. We also assume that there exists a set of points $\left\{x_1, \dots, x_n \right\}$ that are $\gtrsim n^{-1/d}$ separated from each other (this obviously holds for $(0,1)^d$). Let $0 < \varepsilon \ll n^{-1/d}$ and consider the function
$$ f(x) = - c_n + \sum_{k=1}^{n}{ \frac{\varepsilon^{-d}}{n} \chi_{|x-x_k| \leq \varepsilon}}\, ,$$
where $c_n > 0$ is chosen so that  $\int _{\Omega} f=0$.  Since $c_n \sim 1$, we see that
$$ \|f\|_{L^1} \sim 1 \qquad \mbox{and} \qquad \|f\|_{L^{\infty}} \sim \varepsilon^{- d} n^{-1}.$$
Moreover, the zero set is the union of $n$ disjoint spheres, so
$$ \mathcal{H}^{d-1}\left\{x \in \Omega: f(x) = 0\right\} \sim_n n \varepsilon^{d-1}.$$
The transport cost is also easily estimated since most mass has to be transported at least distance $\gtrsim n^{-1/d}$, the minimal spacing between the points. Thus
$$ W_1(f_+, f_-) \cdot \mathcal{H}^{d-1}\left\{x \in (0,1)^d: f(x) = 0 \right\} \sim \|f\|_{L^1} n^{-1/d} n \varepsilon^{d-1} \sim n^{\frac{d-1}{d}} \varepsilon^{d-1} \, ,$$
while
$$ \frac{\|f\|_{L^1}}{\|f\|_{L^{\infty}}} \sim \frac{1}{\|f\|_{L^{\infty}}} \sim  n \cdot \varepsilon^{d} \, .$$
Comparing the two equations above establishes the desired result.
\end{proof}

\section{Proof Theorem 3}\label{sec:upbd_pf}

\begin{proof} This proof is a straightforward consequence of an inequality that was proved in \cite{steiner2018wass, stein2}: for functions
$f:[0,1]^d \rightarrow \mathbb{R}$ satisfying the assumptions of Theorem 3, we have the estimate
$$ W_1( f_+, f_-) \lesssim_{d} \frac{\log{\lambda}}{\sqrt{\lambda}} \|f\|_{L^1}.$$
If we combine this with our Theorem 1, the desired consequence follows immediately. For sake of completeness, we quickly sketch the main idea behind this upper bound (and refer to \cite{steiner2018wass, stein2} for details). Orthogonality to eigenfunctions up to eigenvalue $\lambda$ implies, by the spectral theorem, rapid decay of the (Neumann-)heat evolution
$$ \left\| e^{t\Delta}f\right\|_{L^1} \leq  \left\| e^{t\Delta}f\right\|_{L^2} \leq e^{-\lambda t} \|f\|_{L^2} \leq e^{-\lambda t}  \left\| f\right\|^{1/2}_{L^1} \|f\|_{L^{\infty}}^{1/2}.$$
At the same time, the heat kernel representation of the evolution of the Neumann heat equation 
$$ \left(e^{t\Delta}f\right)(x) = \int_{\Omega}{p_t(x,y) f(y) dy}$$
coupled with the classical decay estimates on the heat equation imply that the heat equation may be understood as one particular transport plan whose cost we can control. The result is sharp up to the logarithm.
\end{proof}

\section{ Further Remarks}
\subsection{The Discrete Case.} We believe that it might be of interest to study
the discrete setting as well, especially in light of possible applications such as network transportation problems. We now describe one 
way of phrasing the question. \\

Let $G=(V,E)$ be a finite, connected, undirected graph and let $f:V \rightarrow \mathbb{R}$ be a function with mean value 0. We can define the Wasserstein distance $W_1(f_+, f_-)$ as before. Indeed, for practical applications one would presumably consider weighted and directed edges in which case $W_1$ can be defined just as well (indeed, the definition is somewhat easier than in the continuous setting because everything is discrete). We define the boundary of a subset $A \subset V$ as the number of edges going from an element in the set to an element in the complement
$$ |\partial A| = \left\{ (u,v) \in E: u \in A ~\mbox{and}~ v \in V\setminus A \right\}.$$
\begin{quote}
\textbf{Basic principle.} The transport cost $W_1(f_+, f_-)$ and the boundary $\partial \left\{v \in V: f(v) > 0\right\}$ cannot both be simultaneously small.
\end{quote}
The precise form of the basic principle will, of course, depend on the actual underlying graph and we believe it would be quite interesting to get a better understanding of this idea. 
\begin{center}
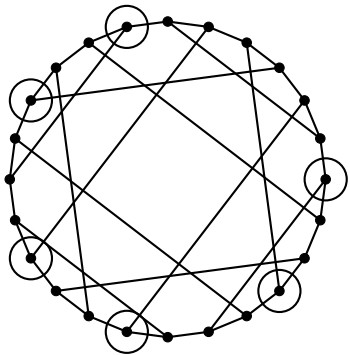
\begin{figure}[h!]
  \begin{tikzpicture}[scale=0.7]
\foreach \a in {1,2,...,24}{
\filldraw (\a*360/24: 3cm) circle (0.09cm);
};
\foreach \a in {1,2,...,24}{
\draw [thick] (\a*360/24: 3cm) --  (\a*360/24 + 360/24: 3cm);
};
\draw [thick] (1*360/24: 3cm) -- (6*360/24: 3cm);
\draw [thick] (2*360/24: 3cm) -- (17*360/24: 3cm);
\draw [thick] (3*360/24: 3cm) -- (10*360/24: 3cm);
\draw [thick] (4*360/24: 3cm) -- (21*360/24: 3cm);
\draw [thick] (5*360/24: 3cm) -- (14*360/24: 3cm);
\draw [thick] (7*360/24: 3cm) -- (12*360/24: 3cm);
\draw [thick] (8*360/24: 3cm) -- (23*360/24: 3cm);
\draw [thick] (9*360/24: 3cm) -- (16*360/24: 3cm);
\draw [thick] (11*360/24: 3cm) -- (20*360/24: 3cm);
\draw [thick] (13*360/24: 3cm) -- (18*360/24: 3cm);
\draw [thick] (15*360/24: 3cm) -- (22*360/24: 3cm);
\draw [thick] (19*360/24: 3cm) -- (24*360/24: 3cm);
\draw [thick] (7*360/24: 3cm) circle (0.4cm);
\draw [thick] (10*360/24: 3cm) circle (0.4cm);
\draw[thick] (14*360/24: 3cm) circle (0.4cm);
\draw [thick] (17*360/24: 3cm) circle (0.4cm);
\draw[thick] (21*360/24: 3cm) circle (0.4cm);
\draw [thick] (24*360/24: 3cm) circle (0.4cm);
   \end{tikzpicture}
\caption{The Nauru Graph on 24 vertices: a subset of 6 vertices such that its characteristic function is orthogonal to the first 19 nontrivial Laplacian eigenfunctions. This subset is very well spread throughout the Graph: is it extremal for the Wasserstein uncertainty principle described below?}
\end{figure}
\end{center}

\subsection{Motivating a Spectral Proof} However, we do have a candidate for extremal sets. The second author has shown \cite{steiner2018wass} that if $f$ is orthogonal to the first few eigenfunctions of the Laplacian, then the heat flow becomes a reasonable transport plan (indeed, at most logarithmic factors away from optimal). Graphical designs \cite{graphdes}, on the other hand, are sets of vertices such that the associated characteristic function is indeed orthogonal to the first few Laplacian eigenfunctions. This seeming relation between the two notions motivates a natural question that we can phrase in a variety of ways
\begin{enumerate}
\item If an optimal way to design locations of supply and demand (e.g., determine where to build the supermarkets) minimizes the Wasserstein distance to the uniform distribution, are graphical designs optimal or near-optimal?
\item Are graphical designs natural candidates for extremal sets for uncertainty principles of the form
$$W_1(f_+, f_-)  \cdot \left|\partial \left\{v \in V: f(v) > 0\right\}\right| \gtrsim_G 1?$$
We emphasize that we do not know what a suitable algebraic form of such an uncertainty principle could be.
\end{enumerate}

\begin{center}
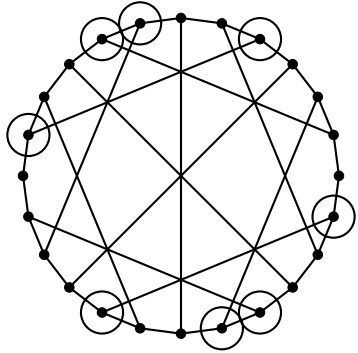
\begin{figure}[h!]
  \begin{tikzpicture}[scale=0.7]
\foreach \a in {1,2,...,24}{
\filldraw (\a*360/24: 3cm) circle (0.09cm);
};
\foreach \a in {1,2,...,24}{
\draw [thick] (\a*360/24: 3cm) --  (\a*360/24 + 360/24: 3cm);
};
\draw [thick] (1*360/24: 3cm) -- (8*360/24: 3cm);
\draw [thick] (2*360/24: 3cm) -- (19*360/24: 3cm);
\draw [thick] (3*360/24: 3cm) -- (15*360/24: 3cm);
\draw [thick] (4*360/24: 3cm) -- (11*360/24: 3cm);
\draw [thick] (5*360/24: 3cm) -- (22*360/24: 3cm);
\draw [thick] (6*360/24: 3cm) -- (18*360/24: 3cm);
\draw [thick] (7*360/24: 3cm) -- (14*360/24: 3cm);
\draw [thick] (9*360/24: 3cm) -- (21*360/24: 3cm);
\draw [thick] (10*360/24: 3cm) -- (17*360/24: 3cm);
\draw [thick] (13*360/24: 3cm) -- (20*360/24: 3cm);
\draw [thick] (16*360/24: 3cm) -- (23*360/24: 3cm);
\draw [thick] (4*360/24: 3cm) circle (0.4cm);
\draw [thick] (7*360/24: 3cm) circle (0.4cm);
\draw [thick] (8*360/24: 3cm) circle (0.4cm);
\draw[thick] (11*360/24: 3cm) circle (0.4cm);
\draw [thick] (16*360/24: 3cm) circle (0.4cm);
\draw[thick] (19*360/24: 3cm) circle (0.4cm);
\draw[thick] (20*360/24: 3cm) circle (0.4cm);
\draw [thick] (23*360/24: 3cm) circle (0.4cm);
   \end{tikzpicture}
\caption{The McGee Graph on 24 vertices: a subset of 8 vertices is orthogonal to the first 21 eigenfunctions exactly. Every other vertex is exactly distance 1 away from exactly one element of this subset. Is it an extremal set?}
\end{figure}
\end{center}

While these questions seem out of reach on general graphs, we can study them in the continuous setting: are point sets that are orthogonal to many Laplacian eigenfunctions candidates for extremal sets? In the continuous setting, this is indeed the case, and a purely spectral proof of the Proposition is possible. Similar arguments are then possible on rather general graphs: the missing ingredient is a graph counterpart of the celebrated result of Bondarenko, Radchenko and Viazovska~\cite{bond, bond2}~establishing the existence of such points (or designs) on $\mathbb{S}^d$. The analogous problem in the case of graphs is wide open \cite{graphdes}, for Riemannian manifolds we refer to~\cite{bilyk, gariboldi}.

\subsection{A Spectral Proof of the Proposition.}
\begin{proof}[A Spectral Proof.]  The example is adapted from \cite {stein2}. The constant $c_d$ will denote a universal constant depending only on the dimension and will change its value from line to line. We fix $\mathbb{S}^{d}$ and $n$ and pick $\left\{x_1, \dots, x_n\right\} \subset \mathbb{S}^{d}$ in such a way that $\|x_i - x_j\| \geq c_d n^{-1/d}$ and
$$ \frac{1}{n} \sum_{k=1}^{n}{p(x_k)} = \frac{1}{\mathcal{H}^{d}(\mathbb{S}^{d})} \int_{\mathbb{S}^{d-1}}{ p(x) d\mathcal{H}^{d}}$$
for all polynomials $p:\mathbb{R}^{d} \rightarrow \mathbb{R}$ of degree less than $c_d n^{1/d}$. The existence of these points follows from \cite{bond, bond2}. We will now introduce the (signed) measure
$$ \mu = -\frac{1}{\mathcal{H}^{d}(\mathbb{S}^{d})} + \frac{1}{n} \sum_{k=1}^{n}{\delta_{x_k}}$$
and will consider the evolution of the heat equation for short time $t$. We note that~$e^{t\Delta}\mu$ is a smooth function for all $t>0$. Moreover, using the standard asymptotic behavior of the heat kernel together with the condition $\|x_i - x_j\| \geq c_d n^{-1/d}$, we see that, for $t^{d/2} \ll n^{-1}$
$$ (e^{t\Delta} \mu) (x) > 0 \qquad \mbox{as long as} \qquad \min_{1 \leq k \leq n} \| x-x_k\| \leq c_d\sqrt{t}.$$
That allows us to estimate
$$ \mathcal{H}^{d-1} \left\{ x \in \mathbb{S}^{d-1}: (e^{t\Delta} \mu) (x) > 0\right\} \sim t^{\frac{d-1}{2}} n.$$
Moreover, we obtain, for $t \lesssim n^{-2/d}$,
$$ \|f\|_{L^{\infty}} \sim \frac{1}{n t^{d/2}}$$
and 
$$ \|f\|_{L^1} \sim 1.$$ 
Using the inequality (see \cite{steiner2018wass})
$$ W_1(f_+, f_-) \lesssim_{\Omega} e^{- n^{2/d} t} \|f\|_{L^1},$$
we obtain
$$ W_1(e^{t\Delta} \mu_+, e^{t\Delta} \mu_-) \lesssim  \frac{\sqrt{\log{n^{2/d}}}}{\sqrt{n^{2/d}}} \lesssim_d \frac{\log{n}}{n^{1/d}}.$$
This shows
$$ \mathcal{H}^{d-1} \left\{ x \in \mathbb{S}^{d-1}: (e^{t\Delta} \mu) (x) > 0\right\} \gtrsim n^{1/d} (nt^{d/2})^{\alpha}$$
which naturally suggests $\alpha = (d-1)/d$ as a candidate for the endpoint.
\end{proof}

\subsection{ Domains without uncertainty principles.} 
We illustrate that the validity of an uncertainty principle as in Theorem 1 requires at least some form of the isoperimetric inequality to be valid. More precisely, let $\Omega \subset \mathbb{R}^2$ be bounded. Then, using the trivial estimate
$$ W_1(f+, f_-) \leq \mbox{diam}(\Omega) \|f\|_{L^1}$$
\begin{center}
\begin{figure}[h!]
\begin{tikzpicture}
   \draw [thick,domain=30:330] plot ({cos(\x)}, {sin(\x)});
      \draw [thick,domain=10:135] plot ({1.3+0.65*cos(\x)}, {0.65*sin(\x)});
      \draw [thick,domain=180+45:350] plot ({1.3+0.65*cos(\x)}, {0.65*sin(\x)});
      \draw [thick,domain=10:165] plot ({2.35+0.45*cos(\x)}, {0.45*sin(\x)});
      \draw [thick,domain=195:350] plot ({2.35+0.45*cos(\x)}, {0.45*sin(\x)});
      \node at (3,0) {$\dots$};
\end{tikzpicture}
\caption{Balls glued together along shrinking interfaces. No nontrivial counterpart for Theorem 1 holds on this domain. }
\end{figure}
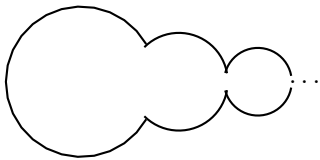
\end{center}
and considering a difference of two
characteristic functions $f=\chi_A-\chi_B$ for two disjoint and equal-sized $A, B \subset \Omega$, we see that any estimate of the type
$$ W_1(f_+, f_-) \cdot \mathcal{H}^{d-1}\left\{x \in (0,1)^d: f(x) = 0 \right\} \gtrsim_{d}  \left( \frac{\|f\|_{L^1}}{\|f\|_{L^{\infty}}} \right)^{\alpha} \|f\|_{L^1} \, .$$
needs, in particular, to also imply
$$ \mathcal{H}^{d-1}(\Omega \cap \partial A) \gtrsim |A|^{\alpha},$$
which is a relative isoperimetric inequality whose validity certainly depends on $\Omega$. Fig. 5 is a simple example of a domain for which either no such estimate is possible or may require arbitrarily large values of $\alpha$ (depending on the speed with which the interfaces decay).

\end{document}